\documentclass{article}
\usepackage{amsmath,amssymb,color,graphicx,enumerate,amsthm,setspace,wrapfig,mathtools,adjustbox,tikz,tikz-cd, quiver, braket,bbm, authblk}
\usepackage{hyperref}
\hypersetup{
    colorlinks=true,
    linkcolor=blue,
    filecolor=magenta,      
    urlcolor=cyan,
    pdfpagemode=FullScreen,
    citecolor=blue
    }

\newcommand{\Z}{\mathbb{Z}}
\newcommand{\C}{\mathbb{C}}

\newcommand{\PP}{\mathbb{P}}

\newcommand{\f}{\frac}

\DeclareMathOperator{\csm}{csm}
\DeclareMathOperator{\st}{stair}

\DeclareMathOperator{\Hom}{Hom}
\DeclareMathOperator{\rk}{rk}

\newtheorem{theorem}{Theorem}[section]
\newtheorem{theorem*}{Theorem}
\newtheorem{remark}{Remark}[section]
\newtheorem{corollary}[theorem]{Corollary}

\newtheorem{lemma}{Lemma}[section]
\newtheorem{proposition}{Proposition}[section]

\newtheorem{definition}{Definition}[section]
\newtheorem{example}{Example}[section]

\title{A ``Staircase'' formula for the Chern-Schwartz-MacPherson cycle of a matroid}
\author[1]{Franquiz Caraballo Alba}
\author[2]{Jeffery Liu}
\affil[1,2]{Florida State University}

\date{September 2024}

\begin{document}

\maketitle

\begin{abstract}
    We provide a formula for the Poincar\'e dual of the Chern-Schwartz-MacPherson (CSM) cycle of a matroid in the Chow ring of the matroid. We derive the formula from the $\mathbb{C}$-realizable case and prove that it satisfies a contraction-deletion formula. From this fact, we prove it holds for all matroids, confirming a conjecture of Fife and Rinc\'on.
\end{abstract}

\section{Introduction}

The Chern-Schwartz-MacPherson (CSM) class generalizes the Chern class of the tangent bundle of a variety. In particular, it is defined even when the variety is singular or non-proper. One interpretation of the CSM class of a variety is in terms of a natural transformation between the functor of constructible functions on varieties and the Chow group functor, for more details and history, see this survey by Brasselet \cite{B00}. In the case where $V$ is the complement of a hyperplane arrangement $\mathcal{A}$ in $\PP^d$, a combinatorial formula for $c_{SM}(V)$ in terms of the characteristic polynomial of the underlying matroid was given separately by Aluffi in \cite{A12} and Huh in \cite{H13}.

De Concini and Procesi defined the notion of a wonderful model of a hyperplane arrangement with respect to a building set in \cite{DP95}. In this paper, we will use the term wonderful model of $\mathcal{A}$ to refer to the wonderful model with respect to the maximal building set. Later, Feichtner and Yuzhvinsky \cite{FY04} define the Chow ring $A^*(M)$ of a matroid $M$ following the ideas of \cite{DP95} and the correspondence between the Chow ring of a nonsingular proper toric variety $X$ and Minkowski weights $MW_*(\Sigma)$ on its associated fan $\Sigma$ proven by Fulton and Strumfels in \cite{FS97}. In \cite{AHK}, Adiprasito, Huh and Katz bring these concepts together and prove that $A^*(M)$ and $MW_*(B(M))$ are dual to each other, in analogy to Poincar\'e duality for the toric case and the realizable case.

In their paper \cite{LdMRS19}, L\'opez de Medrano et al. define $\csm(M)$ for a matroid $M$ and prove in Theorem 3.1 that when $M$ is realizable over $\mathbb{C}$, its CSM cycle corresponds to the geometric CSM class $c_{SM}(\mathcal{W} \setminus \mathcal{A})$ of the complement of the arrangement in the wonderful model under the identification of $MW_*(B(M)) \cong A_*(\mathcal{W})$. They also generalize Aluffi and Huh's statement about the relation between the CSM class and the characteristic polynomial to a general matroid. A separate proof of this relationship was provided by Ashraf and Backman in \cite{AB22} using through an understanding of Tutte activities.

However, in the geometric setting $c_{SM}(\mathcal{W} \setminus \mathcal{A})$ should be Poincar\'e dual to some element $c^{SM}(\mathcal{W} \setminus \mathcal{A}) \in A^*(\mathcal{W})$, since $\mathcal{W}$ is proper and nonsingular. In the combinatorial setting, the Minkowski weights on the Bergman fan of $M$ and the Chow ring of $M$ satisfy Poincar\'e duality, so there should be a class $c^{SM}(M) \in A^*(M)$ with $c^{SM}(M) \cap 1_M = \csm(M)$. In Section 3 we prove:

\setcounter{section}{3}

\begin{theorem}
    The CSM class of the complement of the arrangement $\mathcal{A}$ in its wonderful model $\mathcal{W}$ is given in terms of the Chow ring generators by the formula 
    \begin{align*}
        c_{SM}(\mathcal{W} \setminus \mathcal{A}) &= \prod_{r=1}^{\rk(E)}\left(1 - \sum_{\substack{F \in \mathcal{L}\\ \rk(F) \geq r}} x_F\right) \cap [\mathcal{W}],
    \end{align*}
    where $\mathcal{L}$ is the lattice of non-empty flats of the matroid associated to $\mathcal{A}$ and $x_F$ is the element of the Chow ring of $\mathcal{W}$ associated to the flat $F$.
\end{theorem}

\setcounter{section}{1}

In Mannino's master thesis \cite{Man22}, Conjecture 5.2.4, attributed to a private communication between Tara Fife and Felipe Rinc\'on, proposes that the degree $k$ part of $c^{SM}(M)$ is given by
\begin{align*}
    \sum_{|\mathcal{F}|=k} C_{\mathcal{F}} x_{F_{r_1}}\dots x_{F_{r_k}}
\end{align*}
indexed over all size $k$ multisubset chains $\mathcal{F} = \{F_{r_1} \subseteq \dots \subseteq F_{r_k}\}$ of proper non-empty flats of $M$, with coefficients $C_{\mathcal{F}}$ on each monomial 
\begin{align*}
    C_{\mathcal{F}} := (-1)^k \f{\prod_{i = 1}^k (r_i - i + 1)}{\prod_{r = 1}^{\rk(E)} m_\mathcal{F}(r)!} = (-1)^k \f{(r_1) (r_2 - 1) \cdots (r_k - k + 1)}{m_\mathcal{F}(1)! m_\mathcal{F}(2)! \cdots m_\mathcal{F}(\rk(E))!}.
\end{align*}
where $r_i = \rk(F_{r_i})$ and $m_\mathcal{F}(r)$ counts the multiplicity of $r$ among $\{r_1, \dots, r_k\}$.

Theorem \ref{thm:csm_realizable} proves this conjecture in the case of matroids realizable over $\mathbb{C}$ and Theorem \ref{thm:st=csm} generalizes the answer to arbitrary matroids.

In Section 2 we recall the required combinatorial and algebro-geometric background for the rest of the paper. In Section 3 we prove that our formula does indeed compute $c^{SM}(\mathcal{W} \setminus \mathcal{A})$ in the case where $M$ is realized by the hyperplane arrangement $\mathcal{A}$. In Section 4 we define 
\begin{align*}
    \st(M) := \prod_{r=1}^{\rk(E)}\left(1 - \sum_{\substack{F \in \mathcal{L}\\ \rk(F) \geq r}} x_F\right) \in A^*(M),
\end{align*}
the staircase class of a matroid, which generalizes the formula provided in Section 3 to the case where $M$ is an arbitrary matroid. Interestingly, the staircase class satisfies a contraction-deletion formula as mentioned in Theorem \ref{prop:stair_cd}, similar to that satisfied by Chern-Schwartz-MacPherson cycles. Since the proof of Theorem \ref{prop:stair_cd} is rather involved, the proof is contained in Section \ref{sec5} instead of immediately after the statement of the theorem. The name ``staircase" comes from the visual structure of the class when written out explicitly; in the case where $M = U_{3,3}$ with underlying set $E = \{0,1,2\}$,
\begin{alignat*}{5}
    \st(M) = & (1 & - x_0 - x_1 - x_2 & - x_{01} - x_{02} - x_{12} & - x_{012}) \\
    \cdot & (1 &  & - x_{01} - x_{02} - x_{12} & - x_{012}) \\
    \cdot & (1 &  &  & - x_{012})
\end{alignat*}

Finally, we prove
\setcounter{section}{4}
\setcounter{theorem}{1}

\begin{theorem}\label{thm:st=csm}
    For an arbitrary matroid $M$, $\st(M) \cap 1_M = \csm(M)$.
\end{theorem}
The proof of this statement uses the contraction-deletion formula, the interpretation of the Chow ring in terms of piecewise polynomials \cite{B96} and induction on the number of elements of $M$, which completes the proof of Conjecture 5.2.4, Corollary 7.0.1 and Proposition 7.0.4 in \cite{Man22}.

\textbf{Other work:} In \cite{R23}, Rau provides a formula for the CSM cycle of a matroid $M$ in terms of the tropical polynomials cutting out the Bergman fan $1_M$, which is similar to Lemma \ref{lemma:complement_formula} if we interpret the tropical polynomials as the sections whose zero loci are the components of the pullback of the hyperplane arrangement in the realizable case. In \cite{BEST23}, Berget, Eur, Spink and Tseng construct the framework of tautological classes of matroids, which generalizes the class of the universal subbundle and quotient bundle of the Grassmannian to an arbitrary matroid. Through this framework they obtain a formula for the inclusion of the group of Minkowski weights of a matroid $M$ into the Chow ring of the corresponding permutahedral variety and show that the Poincar\'e dual of $\csm(M)$ in the Chow ring of the permutahedral variety is $c(\mathcal{S}_M)c_{|E|-\rk(M)}(\mathcal{Q}_M)$. Although our proof of Theorem \ref{thm:st=csm} does not use this framework, our result translates to the fact that for classes in $c_{|E|-\rk(M)}(\mathcal{Q}_M) A^*(X_E)$, $c(\mathcal{S}_M)$ behaves as if its Chern roots are $\left\{- \sum_{\substack{F \in \mathcal{L}\\ \rk(F) \geq r}} z_F | 1 \leq r \leq \rk(M)\right\}$ in $A^*(X_E)$, via appendix III in \cite{BEST23}. For specific details, see Section \ref{sec4.3}.

\setcounter{section}{1}
\setcounter{theorem}{0}
\section{Background}

\begin{definition}[Matroid] \label{def:matroid}
    A matroid $M$ is a finite set $E$, together with a non-negative interger valued rank function on its subsets $\rk: 2^E \rightarrow \Z^{\geq 0}$, satisfying the properties:
    \begin{itemize}
        \item for all subsets $A \subseteq E$, $\rk(A) \leq |A|$.
        \item for all $A,B \subseteq E$, if $A\subseteq B$, then $\rk(A) \leq \rk(B)$.
        \item for all $A,B \subseteq E$, $\rk(A\cup B) + \rk(A \cap B) \leq \rk(A) + \rk(B)$
    \end{itemize}
    A subset $F \subseteq E$ is called a flat of the matroid if for all $i \in E \setminus F$, $\rk(F) < \rk(F \cup \{i\})$, or equivalently if it is maximal with respect to inclusion among other subsets of the same rank $\rk(F)$.
    The set of all flats forms a lattice under inclusion. The set of all non-empty flats will be denoted by $\mathcal{L}(M)$, and the set of all proper, non-empty flats will be denoted by $\hat{\mathcal{L}}(M)$. We will also write just $\mathcal{L}$ and $\hat{\mathcal{L}}$ to save space if the matroid is understood.
\end{definition}

\begin{definition}[Loops, Coloops, Simple]
    Given a matroid $M=(E,\rk)$,
    \begin{itemize}
        \item an element $i \in E$ is a loop if $\rk(i) = 0$
        \item an element $i \in E$ is a coloop if every minimal set with rank $\rk(E)$ contains $i$. In other words, $i$ is a coloop if every basis of $M$ contains $i$.
        \item A matroid $M=(E,\rk)$ is simple if its rank $1$ flats are precisely the singletons $\{i\}$, for all $i \in E$.
    \end{itemize}
\end{definition}

\begin{definition}[$\C$-realizable]\label{def:Crealizable}
    Given a collection of hyperplanes $\mathcal{A} = \{H_0, \dots, H_n\}$ in $\PP^d_\C$ (or hyperplanes through the origin in $\C^{d+1}$), the function $\rk: 2^{\{0, \dots, n\}} \rightarrow \Z^{\geq 0}$ defined by $\rk(S) = \mathrm{codim}( \bigcap_{i \in S} H_i)$ is a matroid on the ground set $\{0, \dots, n\}$.

    Given a set of vectors $\{v_0, \dots, v_n\}$ in $\C^{d+1}$, the rank function $\rk(S) = \dim(\mathrm{span}\{v_i \mid i \in S\})$ defines a matroid. This yields the same matroid as the previous example by taking the normal hyperplane to each vector.

    Matroids arising in this way from a hyperplane arrangement in $\PP^d_\C$ (or a vector arrangement in $\C^{d+1}$) are called $\C$-realizable. Note that there exist matroids which are not realizable.
\end{definition}

In the realizable case, if $M = (\{0,\dots,k\},\rk)$ is represented by the set of vectors $\{v_0, \dots, v_k\}$ in $\mathbb{C}^n$, then $M \setminus i$ is represented by the set of vectors $\{v_0, \dots, v_k\} \setminus \{v_i\}$ in $\mathbb{C}^n$ and $M / i$ is represented by $\{v_0 + \braket{v_i}, \dots, v_k + \braket{v_i}\} \setminus \{\braket{v_i}\}$ in $\mathbb{C}^n / \braket{v_i}$.
    
\begin{remark}
    Given a finite set $E$, the matroid $M = (E,|\_|)$ is realizable over $\mathbb{C}$ by $\mathcal{A} = \{V(x_i) | i \in E\}$ in $\PP(\mathbb{C}^E)$.
\end{remark}

\begin{example}
    \label{ex:quasitrivial3}
    Consider the arrangement of four lines in $\PP^2$:
    \begin{align*}
        H_0 &= V(x_0)\\
        H_1 &= V(x_1)\\
        H_2 &= V(x_1 - x_2)\\
        H_3 &= V(x_2)
    \end{align*}
    \begin{figure}
        \centering
        \begin{tikzpicture}
            \draw (-2,0) node[left] {$H_0$} -- (2,0);
            \draw (-2,-2) -- (1,4) node[above] {$H_3$};
            \draw (0,-2) -- (0,4) node[above] {$H_2$};
            \draw (2,-2) -- (-1,4) node[above] {$H_1$};
        \end{tikzpicture}
        \caption{Line Arrangement in $\PP^2$.}
        \label{fig:arrangementExample}
    \end{figure}
    This arrangement is depicted in Figure \ref{fig:arrangementExample} as viewed from the chart $x_0 + x_1 + x_2 \neq 0$.
    The rank function of the corresponding realizable matroid has values
    \begin{align*}
        \rk(\emptyset)&=0\\
        \rk(0)=\rk(1)=\rk(2)=\rk(3)&=1\\
        \rk(01)=\rk(02)=\rk(03)=\rk(12)=\rk(13)=\rk(23)=\rk(123)&=2\\
        \rk(012)=\rk(013)=\rk(023)=\rk(0123)&=3
    \end{align*}
    and its flats has the lattice structure
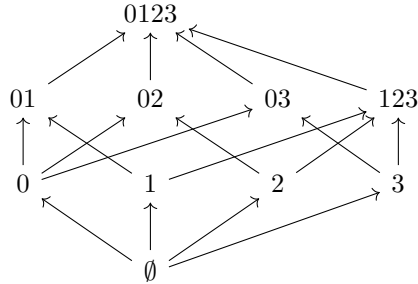
\[\begin{tikzcd}
	& 0123 \\
	01 & 02 & 03 & 123 \\
	0 & 1 & 2 & 3 \\
	& \emptyset
	\arrow[from=2-1, to=1-2]
	\arrow[from=2-2, to=1-2]
	\arrow[from=2-3, to=1-2]
	\arrow[from=2-4, to=1-2]
	\arrow[from=3-1, to=2-1]
	\arrow[from=3-1, to=2-2]
	\arrow[from=3-1, to=2-3]
	\arrow[from=3-2, to=2-1]
	\arrow[from=3-2, to=2-4]
	\arrow[from=3-3, to=2-2]
	\arrow[from=3-3, to=2-4]
	\arrow[from=3-4, to=2-3]
	\arrow[from=3-4, to=2-4]
	\arrow[from=4-2, to=3-1]
	\arrow[from=4-2, to=3-2]
	\arrow[from=4-2, to=3-3]
	\arrow[from=4-2, to=3-4]
\end{tikzcd}\]
\end{example}

For an in depth look at matroids, we reference \cite{W86}.

\begin{definition}[Wonderful model]\label{def:wonderfulmodel}
    Let $\mathcal{A} = \{H_0, \dots, H_n\}$ be hyperplane arrangement in $\PP^d$. For every flat $F$ in the matroid associated to $\mathcal{A}$, there is a linear subvariety corresponding to the intersection $L_F = \cap_{i \in F} H_i$. 
    For an arbitrary arrangement $\mathcal{A}$, its De Concini-Procesi wonderful model $\mathcal{W}$ is a variety constructed from $\PP^d$ by a sequence of blow-ups at each the intersections by order of dimension: 
    \begin{itemize}
        \item First blow up the points $\{L_F\mid F \in \hat{\mathcal{L}}, \rk(F) = d\}$.
        \item Then blow up the proper transform of the lines $\{L_F\mid F \in \hat{\mathcal{L}}, \rk(F) = d-1\}$.
        \item Continue in this manner, at each step $i = 0, \dots, d-1$, blowing up the proper transform of the $i$ dimensional linear subspaces $\{L_F\mid F \in \hat{\mathcal{L}}, \rk(F) = d-i\}$.
    \end{itemize}
    
    Let $\pi: \mathcal{W} \rightarrow \PP^d$ denote this blow-up map. $\pi$ is an isomorphism on the complement of the arrangement $C(\mathcal{A})$. There is a correspondence between the proper non-empty flats of the associated matroid to $\mathcal{A}$ and special divisors of $\mathcal{W}$, specifically, those that are (proper transforms of) exceptional divisors of the blowups described in the definition of $\mathcal{W}$. All collections of these divisors intersect transversely if their flats are comparable, and are disjoint otherwise. See \cite{DP95} for details.
\end{definition}

\begin{definition}[Chow Ring of a Matroid] \label{def:CRmatroid}
    Given a matroid $M$ on a ground set $E$, we associate its Chow ring 
    \begin{align*}
        A^*(M) := \f{\Z[x_F \mid F \in \hat{\mathcal{L}}(M)]}{(I_M + J_M)},
    \end{align*}
    where $I_M$ and $J_M$ are the ideals
    \begin{align*}
        I_M &:= (\{x_Fx_{F'} \mid \text{$F$ and $F'$ are not comparable} \})\\
        J_M &:= \left( \left\{\sum_{\substack{F \in \hat{\mathcal{L}} \\ F \ni i}} x_F - \sum_{\substack{F \in \hat{\mathcal{L}}\\ F \ni j}} x_F\mid i,j \in E\right\} \right).
    \end{align*}
    An equivalent presentation over all non-empty flats is 
    \begin{align*}
        A^*(M) := \f{\Z[x_F \mid F \in \mathcal{L}(M)]}{(I_M + J_M)}
    \end{align*}
    where $I_M$ and $J_M$ are the ideals
    \begin{align*}
        I_M &:= (\{x_Fx_{F'} \mid \text{if $F$ and $F'$ are not comparable} \})\\
        J_M &:= \left( \left\{\sum_{\substack{F \in \mathcal{L} \\ F \ni i}} x_F\mid i \in E\right\} \right)
    \end{align*}
    note that the extra generator $x_E$ equals $$x_E = -\sum_{\substack{F \in \hat{\mathcal{L}} \\ F \ni i}} x_F$$ for all $i \in E$. This is analogous to the presentation of the Chow ring of the permutahedral variety mentioned in Appendix III.1 of \cite{BEST23}.
\end{definition}

When a matroid is $\C$-realizable, arising from a hyperplane arrangement $\mathcal{A}$, its Chow ring is isomorphic to the Chow ring of the wonderful model $A^*(\mathcal{W})$ in the geometric sense, with each $x_F$ corresponding to the class of the divisor associated to the proper flat $F$, and $x_E = -\pi^{*}c_1(\mathcal{O}_{\PP^d}(1))$, minus the pullback of the hyperplane class in $\PP^d$.

\begin{example}
    Let $M$ be the matroid in Example \ref{ex:quasitrivial3}. Its Chow ring is presented by the quotient ring
    \begin{align*}
        \f{\Z[x_0, x_1, x_2, x_3, x_{01}, x_{02},x_{03},x_{123},x_{0123}]}{(I_M + J_M)}
    \end{align*}
    with relations given by the ideals
    \begin{align*}
        I_M = (&x_0x_1, x_0x_2, x_0x_3, x_0x_{123}, x_1x_2, x_1x_3, x_1x_{02}, x_1x_{03}, x_2x_3, x_2x_{01}, x_2x_{03},\\ 
        &x_3x_{01},x_3x_{02},x_{01}x_{02}, x_{01}x_{03}, x_{01}x_{123},
        x_{02}x_{03}, x_{02}x_{123},
        x_{03}x_{123})\\
        J_M = (&x_0 + x_{01} + x_{02} + x_{03} + x_{0123}, x_1 + x_{01} + x_{123} + x_{0123},\\ &x_2 + x_{02} + x_{123} + x_{0123}, x_3 + x_{03} + x_{123} + x_{0123})
    \end{align*}
\end{example}

\begin{definition}[Bergman Fan]
    Let $E$ be a finite set and for $S \subset E$ define $e_S$ be the element of $\mathbb{R}^E$ given by $e_S(i) = 1$ for all $i \in S$. Given a matroid $M = (E, \rk)$, we define $\hat{B}(M)$, the affine Bergman fan of $M$ as the fan in $\mathbb{R}^E$ whose cones are of the form
    \begin{align*}
        \hat{\sigma}_\mathcal{F} := \left\{ \sum_{i = 1}^k a_ie_{F_i} \mid a_i \geq 0 \right\},
    \end{align*}
    where $\mathcal{F} = \{\emptyset \subsetneq F_1 \subsetneq \dots \subsetneq F_k \subsetneq E\}$ is a flag of non-empty flats of $M$, containing the maximal flat $E$. We define the image of this fan under the projection $\mathbb{R}^E \to \mathbb{R}^E / \mathbb{R} e_E$ to be $B(M)$, the Bergman fan of $M$. Note that this projection maps $\hat{\sigma}_{\mathcal{F} \cup \{E\}}$ to the cone $\sigma_\mathcal{F} = {\left\{ \sum_{i = 1}^k a_ie_{F_i} + \mathbb{R}e_E \mid a_i \geq 0 \right\}}$, where $\mathcal{F}$ is a flag of proper non-empty flats.
\end{definition}

The following definitions of Chow ring and Minkowski weights of a fan $\Sigma$ are explored in detail in Chapter $5$ of \cite{AHK}. We reproduce the basic definitions and facts for the reader:

\begin{definition}[Chow Ring of a Fan]
    Given a unimodular polyhedral fan $\Sigma$ in a latticed vector space $N_\mathbb{R} = N \otimes_\mathbb{Z} \mathbb{R}$ we define the Chow ring of $\Sigma$ to be
    \begin{align*}
        A^*(\Sigma) := & \mathbb{Z}[\{x_\rho | \rho \in \Sigma(1)\}]/(I_\Sigma + J_\Sigma),
    \end{align*}
    where $I_\Sigma$ consists of monomials $x_{\rho} x_{\rho'}$ with $\rho, \rho' \in \Sigma(1)$ such that there is no cone $\sigma$ in $\Sigma$ such that $\rho$ and $\rho'$ are both codimension 1 subcones of $\sigma$ and $J_\Sigma$ consists of linear forms
    \begin{align*}
        \sum_{\rho \in \Sigma(1)} \braket{m,v_\rho} x_\rho,
    \end{align*}
    where $m \in \Hom_\mathbb{Z}(N, \mathbb{Z})$ and $v_\rho$ is the generator of $\rho$ in $N$.
\end{definition}

\begin{remark}\label{rem:iso_affine_not_affine}
    In the specific case where $\Sigma$ is the (affine) Bergman fan of a matroid $M$,
    \begin{align*}
        A^*(\hat{B}(M)) \cong A^*(B(M)) \cong A^*(M)
    \end{align*}
    by the isomorphisms given by $x_{\hat{\sigma}_{F}} \mapsto x_{\sigma_{F}} \mapsto x_F$, for $F \neq E$ and $x_{\hat{\sigma}_E} \mapsto - \sum_{F \ni i} x_{\sigma_F}$.
\end{remark}

\begin{definition}[Weighted Fans/Minkowski Weights]
    A weighted fan in $N_\mathbb{R}$ is a polyhedral fan $\Sigma$, of pure dimension $k$, and a weight function $\omega:\Sigma(k) \to \mathbb{Z}$. We say that $(\Sigma, \omega)$ is balanced if for all $\tau \in \Sigma(k-1)$,
    \begin{align*}
        \sum_{\sigma > \tau} \omega(\sigma) v_{\sigma / \tau} \in \braket{\tau},
    \end{align*}
    where $v_{\sigma / \tau}$ is the generator of $\sigma + \braket{\tau}$ in $N_\mathbb{R} / \braket{\tau}$. Given a fan $\Sigma$ of pure dimension $n$, we define $MW_K(\Sigma)$, the group $k$-dimensional Minkowski weights on $\Sigma$ to be the set of balanced weighted subfans of $\Sigma$ with addition given by addition on the weight functions. We define $MW_*(\Sigma)$ as the direct sum of the groups of $k$-dimensional Minkowski weights.
\end{definition}

\begin{remark}[Poincaré Duality]\label{rem:PD}
    Given a pure $n$-dimensional, balanced, unimodular fan $\Sigma$ in $N_\mathbb{R}$, $A^k(\Sigma) \cong MW_{n - k}(\Sigma)$. We denote the isomorphism by $\_ \cap 1_\Sigma$, where $1_\Sigma$ is the function $\sigma \mapsto 1$ for all $\sigma \in \Sigma(n)$. In the case where $\Sigma = B(M)$ for some matroid $M$, we simply write $1_M$ for $1_{B(M)}$. In particular, Remark \ref{rem:iso_affine_not_affine} and the statements above imply that for all matroids $M$, $MW_*(\hat{B}(M)) \cong MW_*(B(M))$.
\end{remark}

\begin{remark}
    In the case where $\Sigma$ is the fan associated to a proper, nonsingular toric variety $X(\Sigma)$, $A^*(\Sigma) \cong A^*(X(\Sigma))$ and $A_*(X(\Sigma)) \cong MW_*(\Sigma)$, see \cite{FS97} for details.
\end{remark}

\section{A ``Staircase" CSM Class Formula for Realizable Matroids}
The Chern-Schwartz-MacPherson (CSM) class generalizes the Chern class of the tangent bundle of a variety. More explicitly, let $Var_\mathbb{C}$ be the category of varieties with proper maps as morphisms and $Ab$ be the category of abelian groups.

\begin{proposition}
    There exists a unique natural transformation $c_*:\mathcal{C}_* \to A_*$ such that for a proper, nonsingular variety $V$, $c_*(\mathbbm{1}_V) = c(TV) \cap [V]$ where $\mathcal{C}_*: Var_\mathbb{C} \to Ab$ is the functor of constructible functions and $A_*: Var_\mathbb{C} \to Ab$ is the Chow group functor. See Theorem 1 in \cite{M74} for details.
\end{proposition}

\begin{definition}[CSM class]
    For a (possibly singular and/or non-proper) subvariety $X$ embedded in a nonsingular proper variety $V$, define $c_{SM}(X) = c_*(\mathbbm{1}_X) \in A_*(V)$.
\end{definition}

L\'opez de Medrano, Rinc\'on, and Shaw defined an analogous CSM cycle for matroids, and proved that when a matroid is realizable over $\mathbb{C}$, its CSM cycle corresponds to the geometric CSM class $c_{SM}(\mathcal{W} \setminus \mathcal{A})$ of the complement of the arrangement in the wonderful model under the identification of $A_*(M) \cong A_*(\mathcal{W})$ (Theorem 3.1 in \cite{LdMRS19}).

A formula for the CSM cycle of a matroid in terms of the generators of the Chow ring was conjectured by Fife, T. and Rincón, F in private communication, according to Conjecture 5.2.4 in \cite{Man22}.

In this section, we prove such a formula for the CSM class of the complement of the arrangement by geometric means, thus resolving the conjecture for all $\mathbb{C}$-realizable matroids.

In the next section, we extend this formula to all matroids, and prove that it gives the same CSM cycle as defined by L\'opez de Medrano, Rinc\'on, and Shaw.

To proceed, we utilize some geometric formulas for the CSM class due to Aluffi:
\begin{lemma} [Chern Classes of Blow-ups] (Lemma 1.3 in \cite{A09})
    Let $X \subseteq Y$ be nonsingular varieties, and $\pi: \tilde{Y} \rightarrow Y$ the blow-up of Y along $X$, and $\tilde{X}$ the exceptional divisor. If $X$ is a complete intersection of hypersurfaces $Z_1, \dots, Z_d$ meeting transversally in $Y$, then the Chern class of the tangent bundle of the blow up $\tilde{Y}$ is related to $c(T_Y)$ by the formula
    \begin{align*}
        c(T_{\tilde{Y}}) = \f{(1 + \tilde{X})(1 + \pi^*Z_1 - \tilde{X})\cdots(1 + \pi^*Z_d - \tilde{X})}{(1 + \pi^*Z_1)\cdots(1 + \pi^*Z_d)}\pi^*c(T_Y)
    \end{align*}
    \label{lemma:blowup_formula}
\end{lemma}

\begin{lemma} [CSM Classes of Complements] (Theorem 1 in \cite{A99})
    Let $X \subset Y$ be a subvariety of a nonsingular variety, with $X = D_1 \cup \dots \cup D_n$ a normal crossings divisor with smooth components. The CSM class of the complement $Y\setminus X$ is
    \begin{align*}
        c_{SM}(Y\setminus X) &= c(\Omega^1_{Y}(\log X)^\vee ) \cap [Y] \\
        &= \f{c(T_Y)}{(1+D_1) \cdots (1+D_n)} \cap [Y]
    \end{align*}
    \label{lemma:complement_formula}
\end{lemma}

\begin{theorem}
    The CSM class of the complement of the arrangement in its wonderful model is given in terms of the Chow ring generators by the formula 
    \begin{align*}
        c_{SM}(\mathcal{W} \setminus \mathcal{A}) &= \prod_{r=1}^{\rk(E)}\left(1 - \sum_{\substack{F \in \mathcal{L}\\ \rk(F) \geq r}} x_F\right) \cap [\mathcal{W}]
    \end{align*}
    \label{thm:csm_realizable}
\end{theorem}

\begin{proof}

The wonderful model $\mathcal{W}$ is constructed by a sequence of blow-ups of $\PP^d$ at the intersection of hyperplanes determined by each flat $F \in \hat{\mathcal{L}}$. If we blow-up according to order by rank, highest to lowest (i.e. by dimension, lowest to highest), the center of each blow-up is the intersection of $\rk(F)$ many (pullbacks of) hyperplanes determined by $F$, minus the exceptional divisors already blown up in the intersection (corresponding to the flats $F'$ which contain $F$). That is, when all blow-ups are complete, each of the $\rk(F)$ hyperplanes in locus of blow-up pulls back to 
\begin{align*}
    \pi^* h -\sum_{\substack{F' \in \hat{\mathcal{L}} \\ F' \supsetneq F}} x_{F'} &= -x_E -\sum_{\substack{F' \in \hat{\mathcal{L}} \\ F' \supsetneq F}} x_{F'}\\
    &= -\sum_{\substack{F' \in \mathcal{L} \\ F' \supsetneq F}} x_{F'}
\end{align*}
in $A^*(\mathcal{W})$, where $h$ is the class of the hyperplane in $\PP^d$, $c_1(\mathcal{O}_{\PP^d}(1))$. By applying the blow-up formula (Lemma \ref{lemma:blowup_formula}) according to each flat $F \in \hat{\mathcal{L}}$,
\begin{align*}
c(T_\mathcal{W}) = \prod_{F \in \hat{\mathcal{L}}} \left[\f{(1 + x_F) (1 -\sum_{\substack{F' \in \mathcal{L} \\ F' \supsetneq F}} x_{F'} - x_F)^{\rk(F)} }{(1 -\sum_{\substack{F' \in \mathcal{L} \\ F' \supsetneq F}} x_{F'})^{\rk(F)}} \right]\pi^*c(T_{\PP^d}).
\end{align*}

The complement of the arrangement in $\mathcal{W}$ is the complement of the union of all divisors $x_F$ for all $F \in \hat{\mathcal{L}}$, which is normal crossings with smooth components by construction. By the complement formula (Lemma \ref{lemma:complement_formula}), we get
\begin{align*}
    c_{SM}(\mathcal{W} \setminus \mathcal{A}) &= \f{c(T_{\mathcal{W}})}{\prod_{F \in \hat{\mathcal{L}}} (1 + x_F)}  \cap [\mathcal{W}]\\
    &= \prod_{F \in \hat{\mathcal{L}}} \left[\f{(1 + x_F) (1 -\sum_{\substack{F' \in \mathcal{L} \\ F' \supsetneq F}} x_{F'} - x_F)^{\rk(F)} }{(1 + x_F)(1 -\sum_{\substack{F' \in \mathcal{L} \\ F' \supsetneq F}} x_{F'})^{\rk(F)}} \right]\pi^* c(T_{\PP^d}) \cap [\mathcal{W}].
\end{align*}
By cancellation, and $\pi^*c(T_{\PP^d}) = (1 + \pi^* h)^{d+1} = (1 + \pi^* h)^{\rk(E)} = (1-x_E)^{\rk(E)}$, we get
\begin{align*}
   c_{SM}(\mathcal{W} \setminus \mathcal{A}) = \dots &= \prod_{F \in \hat{\mathcal{L}}} \left[\f{(1 -\sum_{\substack{F' \in \mathcal{L} \\ F' \supsetneq F}} x_{F'} - x_F)^{\rk(F)} }{(1 -\sum_{\substack{F' \in \mathcal{L} \\ F' \supsetneq F}} x_{F'})^{\rk(F)}} \right](1 -x_E)^{\rk(E)} \cap [\mathcal{W}]\\
    &= \prod_{F \in \hat{\mathcal{L}}} \left[\f{(1 - \sum_{\substack{F' \in \mathcal{L} \\ F' \supseteq F}} x_{F'})^{\rk(F)} }{(1 -\sum_{\substack{F' \in \mathcal{L} \\ F' \supsetneq F}} x_{F'}))^{\rk(F)}} \right]\f{(1 -x_E)^{\rk(E)}}{(1)^{\rk(E)}} \cap [\mathcal{W}]\\
    &= \prod_{F \in \mathcal{L}} \left(\f{1 - \sum_{\substack{F' \in \mathcal{L} \\ F' \supseteq F}} x_{F'}}{1- \sum_{\substack{F' \in \mathcal{L} \\ F' \supsetneq F}} x_{F'}}\right)^{\rk(F)} \cap [\mathcal{W}].
\end{align*}
Next, we compute the coefficients in the formal expansion of the above expression in $\Z[x_F \mid F \in \mathcal{L}]$. The coefficient of each monomial $\prod_{F \in \mathcal{F}} x_F$ (where $\mathcal{F}$ some multisubset of $\mathcal{L}$) may be extracted as follows:

$\mathcal{F}$ must be a chain, since otherwise  $\prod_{F \in \mathcal{F}} x_F = 0$, since it corresponds to an empty intersection in $\mathcal{W}$. Choose a maximal flag $\{\emptyset = F_0 \subsetneq F_1 \subsetneq \dots \subsetneq F_d \subsetneq F_{d+1} = E\}$ (where $\rk(F_i) = i$) containing the chain $\mathcal{F}$ as a multisubset. We may rewrite the monomial as

$$\prod_{F \in \mathcal{F}} x_F = x_{F_{r_1}} \cdots x_{F_{r_k}}.$$

where $k = |\mathcal{F}|$ is the degree of the monomial (also, the size of the multiset), and $r_1 \leq \dots \leq r_k$ is a monotonic sequence. In other words, we re-index the multiset according to the order in the flag by inclusion: 
$$\mathcal{F} = \{F_{r_1}, \dots, F_{r_k}\}, \quad F_{r_1}\subseteq \dots \subseteq F_{r_k}.$$

Let $m_\mathcal{F}(r) = |\{F \in \mathcal{F} \mid \rk(F) = r\}|$ count the multiplicity of $F_r$ in the multisubset $\mathcal{F}$, i.e. the number of occurrences of $r$ in $\{r_1, \dots, r_k\}$.

For all $F$ not included in the flag, we may set $x_F = 0$. This has no effect on $\prod_{F \in \mathcal{F}} x_F$, and only kills some other monomials. If $F$ is not in the flag, the fraction $\left(\f{1 - \sum_{F' \supseteq F} x_{F'}}{1- \sum_{F' \supsetneq F} x_{F'}}\right)^{\rk(F)}$ becomes $1$. Otherwise, when $F = F_r$ in the flag, it becomes $\left(\f{1 - \sum_{i = r}^{\rk(E)} x_{F_i}}{1- \sum_{i = r+1}^{\rk(E)} x_{F_i}}\right)^r$. The expression then simplifies by cancellation (we indicate the result of killing the non-flag divisors by $\mapsto$):
\begin{align*}
    \prod_{F \in \mathcal{L}} \left(\f{1 - \sum_{F' \supseteq F} x_{F'}}{1- \sum_{F' \supsetneq F} x_{F'}}\right)^{\rk(F)} &\mapsto
     \prod_{r = 1}^{\rk(E)} \left(\f{1 - \sum_{i = r}^{\rk(E)} x_{F_i}}{1- \sum_{i = r+1}^{\rk(E)} x_{F_i}}\right)^r \\
     & = \prod_{r = 1}^{\rk(E)} \left(\f{1 - (x_{F_r} + x_{F_{r+1}} + \dots + x_{E})}{1- (x_{F_{r+1}} + \dots + x_{E})}\right)^r\\
    &= \prod_{r = 1}^{\rk(E)} \left(1 - \sum_{i = r}^{\rk(E)} x_{F_i}\right) \\ 
    &= \prod_{r = 1}^{\rk(E)} (1 - (x_{F_r} + \dots + x_{E})).
\end{align*}
The coefficient of the monomial $x_{F_{r_1}} \cdots x_{F_{r_k}}$ may be deduced by counting how many times it the appears in the expansion\footnote{alternatively, we can use the Taylor's Theorem to compute the coefficient by taking the appropriate partial derivatives at zero.}. The count amounts to choosing from which factor each $x_{F_{r_i}}$ comes, then dividing to account for identical permutations. Note that each $-x_{F_{r_i}}$ appears as a term only in the first $r_i$ factors. There are $r_1$ possible choices for $x_{F_{r_1}}$, then $r_2 - 1$ many remaining choices for $x_{F_{r_2}}$, and so on (i.e. $r_i - i + 1$ choices for $x_{F_{r_i}}$ once the previous $i-1$ variables have been chosen). Therefore, the coefficient on $x_{F_{r_1}} \cdots x_{F_{r_k}}$ is
\begin{align*}
    C_{\mathcal{F}} := (-1)^k \f{\prod_{i = 1}^k (r_i - i + 1)}{\prod_{r = 1}^{\rk(E)} m_{\mathcal{F}}(r)!} = (-1)^k \f{(r_1) (r_2 - 1) \cdots (r_k - k + 1)}{m_{\mathcal{F}}(1)! m_{\mathcal{F}}(2)! \cdots m_{\mathcal{F}}(\rk(E))!}.
\end{align*}
Denote this coefficent by $C_{\mathcal{F}}$ for each multisubset $\mathcal{F}$. These are precisely the coefficients conjectured by Fife and Rincón.
Furthermore, an alternate expression
\begin{align*}
\prod_{r=1}^{\rk(E)}\left(1 - \sum_{\substack{F \in \mathcal{L} \\ \rk(F) \geq r}} x_F\right)
\end{align*}
yields the same coefficients on every monomial, since it has the same form after killing non-flag divisors: $\prod_{r=1}^{\rk(E)}\left(1 - \sum_{\rk(F) \geq r} x_F\right) \mapsto \prod_{r = 1}^{\rk(E)} \left(1 - \sum_{i = r}^{\rk(E)} x_{F_i}\right)$. By comparing coefficients, we conclude that it is equal to the original CSM class:
\begin{align*}
    c_{SM}(\mathcal{W} \setminus \mathcal{A}) = \dots &= \prod_{F \in \mathcal{L}} \left(\f{1 - \sum_{\substack{F \in \mathcal{L} \\ F' \supseteq F}} x_{F'}}{1- \sum_{\substack{F \in \mathcal{L} \\F' \supsetneq F}} x_{F'}}\right)^{\rk(F)} \cap [\mathcal{W}]\\
    &= \sum_{\mathcal{F} \text{ multisubset of } \mathcal{L}} C_{\mathcal{F}} \prod_{F \in \mathcal{F}} x_F \cap [\mathcal{W}]\\
    &= \prod_{r=1}^{\rk(E)}\left(1 - \sum_{\substack{F \in \mathcal{L} \\ \rk(F) \geq r}} x_F\right) \cap [\mathcal{W}].
\end{align*}
\end{proof}

\begin{corollary}
    The formula conjectured by Fife and Rinc\'on holds for $\mathbb{C}$-realizable matroids.
\end{corollary}

\begin{example}
     Let $\mathcal{A}$ be the arrangement in Example \ref{ex:quasitrivial3}. The CSM class of the complement of the arrangement in its wonderful compactification by the previous result is
    \begin{align*}
         c_{SM}(\mathcal{W} \setminus \mathcal{A})= &(1-x_0-x_1-x_2-x_3-x_{01}-x_{02}-x_{03}-x_{123}-x_{0123})\\
         \cdot&(1-x_{01}-x_{02}-x_{03}-x_{123}-x_{0123})
         \cdot(1-x_{0123}) \cap [\mathcal{W}]\\
        = &(1-x_{123}-x_{0123}) \cap [\mathcal{W}].
    \end{align*}
\end{example}

\section{Relation to the CSM cycle of a matroid}

The result from the previous section motivates the following definition for all matroids, not necessarily realizable.
\begin{definition}[]
    Given a loopless matroid $M$, define the class $\st(M)$ in its Chow ring $A^*(M)$ by the formula
    \begin{align*}
        \st(M) := \prod_{r=1}^{\rk(E)}\left(1 - \sum_{\rk(F) \geq r} x_F\right) \in A^*(M).
    \end{align*}
    In the case where $M$ has a loop, define $\st(M) = 0$.
    \label{def:staircase}
\end{definition}
We prove in this section that $\st(M) \cap 1_M$ is the CSM cycle of a matroid defined by L\'opez de Medrano, Rinc\'on, and Shaw. 
Note that the coefficients $C_\mathcal{F}$ conjectured by Fife and Rinc\'on are precisely the coefficients in the expansion of $\st(M)$, so this is equivalent to the statement of their conjecture in full generality.

First we show how this class respects behaves with respect to certain deletion and contraction operations on a matroid.

\begin{definition}[Deletion and Contraction]
    Let $M$ be a matroid on ground set $E$ with rank function $\rk_M$. For any $i \in E$, we can construct the matroid $M\setminus i$, called the deletion of $M$ by the element $i$. This is the matroid on the set $E\setminus i$ with rank function $\rk_{M\setminus i}: 2^{E\setminus i} \to \mathbb{Z}$ given by:
    \begin{align*}
        \rk_{M\setminus i}(S) := \rk_M(S)
    \end{align*}
    For any $i \in E$, we can construct the matroid $M / i$, called the contraction of $M$ by the element $i$.
    This is the matroid on the set $E\setminus i$ with rank function $\rk_{M\setminus i}: 2^{E\setminus i} \to \mathbb{Z}$ given by:
    \begin{align*}
        \rk_{M / i}(S) := \rk_M(S \cup i) - \rk_M(i)
    \end{align*}

\end{definition}

\subsection{Contraction-deletion of $\st$}

For a given element $i \in E$ which is not a coloop, we obtain a projection $\delta: B(M) \to B(M \setminus i)$. By Proposition 2.22 in \cite{S13}, $\delta$ is a tropical modification along some rational function $\phi$; therefore, $\delta$ induces a pullback $\delta^*:Z_*(B(M \setminus i) \to Z_*(B(M))$ and since $div_{B \setminus i}(\phi) = B(M / i)$, we also obtain a pullback $Z_*(B(M / i)) \to Z_*(B(M \setminus i))$, where $Z_*(\Sigma)$ is the group of cycles on the fan $\Sigma$, see Definition 5.12 in \cite{AR09}. These have corresponding ring/group homomorphisms in the context of the Chow rings, see Section 4.2 for details.

\begin{definition}[Deletion map]
    To distinguish the generators of the Chow rings of $M, M \setminus i$ and $M/i$, we label them by
    \begin{align*}
        A^*(M) &= \mathbb{Z}[x_F \mid F \in \mathcal{L}(M)]/(I_{M} + J_{M})\\
        A^*(M\setminus i) &= \mathbb{Z}[y_F \mid F \in \mathcal{L}(M\setminus i)]/(I_{M \setminus i} + J_{M \setminus i})\\
        A^*(M/ i) &= \mathbb{Z}[z_F \mid F \in \mathcal{L}(M/i)]/(I_{M/i} + J_{M/i}).
    \end{align*}
    
    Define the deletion pullback $\bar{\delta}^*: A^*(M \setminus i) \rightarrow A^*(M)$ as the ring homomorphism induced by
    \begin{align*}
        \bar{\delta}^*(y_F) = \sum_{\substack{F' \in \mathcal{L}(M) \\ F'\setminus i = F}} x_{F'}.
    \end{align*}

    Define the element $y_i \in A^*(M \setminus i)$ by
    \begin{align}\label{def:yi}
        y_i := - \sum_{\substack{F \in \mathcal{L}(M\setminus i) \\ F \notin \mathcal{L}(M) \\ F \cup {i} \in \mathcal{L}(M)}} y_F.
    \end{align}

    Multiplication by $y_i$ defines a map $A^*(M / i) \rightarrow A^*(M \setminus i)$
    \begin{align*}
        z \mapsto y_i \cdot \iota(z)
    \end{align*}
    where $\iota(z_F) = y_F$, i.e. $\iota$ replaces the variables $z$ with $y$. This is an embedding $A^*(M / i)$ as the subgroup $y_i A^*(M \setminus i)$. By composition we also obtain a group homomorphism $\bar{\delta}^* (y_i \cdot \iota(\_)): A^*(M / i) \rightarrow A^*(M)$.
\end{definition}

Using these definitions, the staircase class of a matroid satisfies the contraction-deletion formula stated below. For the proof of Theorem \ref{prop:stair_cd}, see Section \ref{sec5}.

\begin{theorem}
    For all non-coloop elements $i \in E$,
    \begin{align*}
        \st(M) = \bar{\delta}^*(\st(M\setminus i)) - \bar{\delta}^* (y_i \cdot \iota(\st(M/i))).
    \end{align*}
    \label{prop:stair_cd}
\end{theorem}

\subsection{Piecewise polynomials and cocycles}

There is an equivalent formulation of $A^*(\Sigma)$ in terms of so-called Courant functions on $\Sigma$. This approach defines $A^*(\Sigma)$ as a quotient of continuous piecewise polynomial functions with integer coefficients on $\Sigma$ by linear functions from $N$ to $\mathbb{R}$, see Section 4 of \cite{AHK} for more details. 

This other formulation leads to the theory of tropical cocycles on $\Sigma$, introduced by François in \cite{F13} and elaborated upon by Gross and Shokrieh in \cite{GS19}, in analogy to Allermann and Rau's theory of cycles on $\Sigma$ from \cite{AR09}. In particular, $Z_*(\Sigma)$, the tropical cycles on $\Sigma$, consists of weighted fans on subdivisions of $\Sigma$, while $C^*(\Sigma)$, the tropical cocycles on $\Sigma$, are continuous functions that are piecewise polynomial in a subdivision of $\Sigma$. In \cite{GS19}, the authors prove that, when given the corresponding intersection products, these $C^*(\Sigma)$-algebras are isomorphic. In particular, given a morphism of fans $f:\Sigma_1 \to \Sigma_2$, a cocycle $\alpha \in C^*(\Sigma_2)$ and a cycle $X \in Z_*(\Sigma_2)$, $f^*(\alpha \cdot X) = f^*(\alpha) \cdot f^*(X)$, where $f^*:Z_*(\Sigma_2) \to Z_*(\Sigma_1)$ is defined in Corollary 4.7 in \cite{GS19}, which agrees with Definition 2.16 in \cite{S13} in the case where $f$ is a tropical modification. 

Using this theory, we describe the pullback $\delta^*:Z_*(B(M \setminus i)) \to Z_*(B(M))$ in terms of cocycles, for a matroid $M$ of rank $r$ on a set $E$ and $i \in E$ not a coloop. Since in our case the map $\delta:B(M) \to B(M \setminus i)$ maps cones to cones, the image of weights on $B(M \setminus i)$ under $\delta^*$ is contained in weights on $B(M)$. Similarly, the image under $\delta^*$ of a piecewise polynomial on $B(M \setminus i)$ is a piecewise polynomial on $B(M)$, giving a pullback $A^*(M \setminus i ) \to A^*(M)$. 

Since $\hat{\delta}:\hat{B}(M) \to \hat{B}(M \setminus i)$ satisfies the conditions of Proposition 3.10 in \cite{FR10}, there exists a rational function $\phi$ such that $\delta$ is a tropical modification along $\phi$, and there is an isomorphism $\phi \cdot MW_*(\hat{B}(M \setminus i)) \xrightarrow{\sim} MW_*(\hat{B}(M / i))$. The inverse of this isomorphism is given by $\hat{j}_*:MW_*(\hat{B}(M / i)) \to MW_*(\hat{B}(M \setminus i))$.

\begin{lemma}
    The inclusion $j_*:MW_*(B(M / i)) \to MW_*(B(M \setminus i))$ is given by $j_*(z \cap 1_{M / i}) = y_i \cdot \iota(z) \cap 1_{M \setminus i}$.
\end{lemma}
\begin{proof}
    By Proposition 3.10 in \cite{FR10}, $\phi$ is given by 
    \begin{align*}
        \phi(v_{\hat{\sigma}_F}) = & \rk_{M / i}(F) -\rk_{M \setminus i}(F) \\
        = & \rk_M(F \cup i) - \rk_M(i) - \rk_M(F) \\
        = & \rk_M(F \cup i) - \rk_M(F) - 1.
    \end{align*}
    If $F$ is not a flat in $M$, but $F \cup i$ is, then $\rk_M(F \cup i) = \rk_M(F)$, thus
    \begin{align*}
        \phi(v_{\hat{\sigma}_F}) = & \rk_M(F \cup i) - \rk_M(F) - 1 \\
        = & \rk_M(F) - \rk_M(F) - 1 \\
        = & -1.
    \end{align*}
    If $F$ is a flat in $M$, then $\rk_M(F \cup i) = \rk_M(F) + 1$, thus
    \begin{align*}
        \phi(v_{\hat{\sigma}_F}) = & \rk_M(F \cup i) - \rk_M(F) - 1 \\
        = & \rk_M(F) + 1 - \rk_M(F) - 1 \\
        = & 0.
    \end{align*}
    Therefore, under the isomorphism $MW_*(\hat{B}(M)) \cong MW_*(B(M))$, $\phi \cdot \_ = y_i \cap \_$, where $y_i$ is defined in \eqref{def:yi}. Finally, it is straightforward check to see that $y_i \cdot \iota(z) \cap 1_{M \setminus i} = j_*(z \cap 1_{M / i})$.
\end{proof}

\begin{lemma}\label{lemma:e=d}
    For cocycles $y \in A^*(M\setminus i )$, $\bar{\delta}^*(y) \cap 1_M = \delta^*(y \cap 1_{M \setminus i})$.
\end{lemma}
\begin{proof}
    Since these are ring homomorphisms and both rings are generated in degree 1, we need only check this holds for Courant functions. Courant functions are uniquely determined by their values on rays; thus the claim simplifies to the statement that for all flats $F \in \hat{\mathcal{L}}(M \setminus i)$, if $F$ is a flat in $M$, then $\delta^{-1}(\sigma_F) = \sigma_{F} \cup \sigma_{F \cup i}$ and $x_F(\bar{e}_F) = x_F(\delta(e_{F\cup i})) = x_F(\delta(e_F))$, and if $F$ is not a flat of $M$ then $\delta^{-1}(\sigma_F) = \sigma_{F \cup i}$ and $x_F(\bar{e}_F) = x_F(\delta(e_{F\cup i}))$. From the definition of $\delta$, it is clear that $\delta(e_{F\cup i}) = \delta(e_F) = \bar{e}_F$, thus proving the claim. 
\end{proof}

\begin{theorem}
    For a matroid $M$, $\st(M) \cap 1_M = \csm(M)$.
\end{theorem}
\begin{proof}
    We proceed by induction. Note that for the loop matroid $M_0$, i.e. the matroid on one element whose only element is a loop, $\csm(M_0) = 0 = 0 \cap 1_{M_0} = \st(M_0) \cap 1_{M_0}$ and for the isthmus $M_1$, i.e. the matroid on one element whose only element is independent, $\csm(M_1) = 1_{M_1}$. Now, assume that for all matroids $M'$ whose ground set consists of $n$ elements $\st(M') \cap 1_{M'} = \csm_k(M')$ and let $M$ be a loopless matroid whose ground set consists of $n + 1$ elements. If all elements of the ground set of $M$ are coloops, then $M$ is representable over $\mathbb{C}$ and thus, by Theorem \ref{thm:csm_realizable}, the conclusion follows. Now assume there exists $i \in M$ that is not a coloop. Then, by Theorem \ref{prop:stair_cd}, $\st(M) = \bar{\delta}^* \st(M \setminus i) - \bar{\delta}^* \st(M / i)$. Then, the ground sets of $M \setminus i$ and $M / i$ contain $n$ elements. Thus, by Lemma \ref{lemma:e=d} and Theorem 5.4 in \cite{LdMRS19},
    \begin{align*}
        \st(M) \cap 1_M = & \\
        = & \bar{\delta}^* (\st(M \setminus i)) \cap 1_M - \bar{\delta}^* (y_i \cdot \iota (\st(M / i))) \cap 1_{M} \\
        = & \delta^* (\st(M \setminus i) \cap 1_{M \setminus i}) - \delta^*( y_i \cdot \iota (\st(M / i)) \cap 1_{M \setminus i})  \\
        = & \delta^*\csm(M \setminus i) - \delta^* j_*(\st(M / i) \cap 1_{M / i}) \\
        = & \delta^*\csm(M \setminus i) - \delta^* j_*\csm(M / i) \\
        = & \csm(M).
    \end{align*}
\end{proof}

Thus, we obtain as a consequence of the computation in Theorem \ref{thm:csm_realizable} and Theorem \ref{thm:st=csm}:

\begin{corollary}\label{cor:conjecture_true}
    The formula of Fife and Rinc\'on holds for all matroids and the proofs of Corollary 7.0.1 and Proposition 7.0.4 in \cite{Man22} are valid.
\end{corollary}

\subsection{Connection to tautological classes of matroids}\label{sec4.3}

For an arbitrary set $E$, with $|E| = n$, the permutohedral variety $X_E$ is the (maximal) wonderful model of the hyperplane arrangement consisting of the coordinate hyperplanes in $\PP^{n-1}$. The matroid associated to this hyperplane arrangement is the uniform matroid $U_{n,n}$ on $E$, for which all subsets are flats. Hence, its Chow ring $A^*(X_E)$ is generated by variables $z_F$ for all subsets $\emptyset \subsetneq F \subset E$

For an arbitrary matroid $M$ on the same set $E$, the authors of \cite{BEST23} define a tautological subbundle $\mathcal{S}_M$ and quotient bundle $\mathcal{Q}_M$ over $X_E$, with Chern classes valued in $A^*(X_E)$.

If we denote by $\iota:A^*(M) \rightarrow A^*(X_E)$ the map $\iota(x_F) = z_F$, Theorem 7.6 in \cite{BEST23} says that there is a canonical embedding $A^*(M) \rightarrow A^*(X_E)$ as the ideal generated by $c_{|E| - \rk(M)}(\mathcal{Q}_M)$, given by $c_{|E| - \rk(M)}(\mathcal{Q}_M) \cdot \iota(\_)$.

Under this canonical identification, in Theorem 8.4 in \cite{BEST23}, the authors prove that the Poincar\'e dual of the CSM class is given by $c(\mathcal{S}_M) \cdot c_{|E|- \rk(M)}(\mathcal{Q}_M)$ and in Appendix III.1, they give the Chern root decomposition for $c(\mathcal{S}_M)$, which yields the same staircase form $\iota(\st(M))$ after removing the terms which are annihilated by $c_{|E|- \rk(M)}(\mathcal{Q}_M)$.

\begin{remark}
    The map $c(\mathcal{S}_M) \cdot \_:A^*(X_E) \to A^*(X_E)$ equals the map $\iota (\st(M)) \cdot \_:A^*(X_E) \to A^*(X_E)$, when restricted to the submodule $c_{|E|-\rk(M)}(Q_M) A^*(X_E)$.
\end{remark}

\section{Proof of Theorem 4.1}\label{sec5}

\begin{proof}
We partition the nonempty flats of $M$ according to the flowchart:
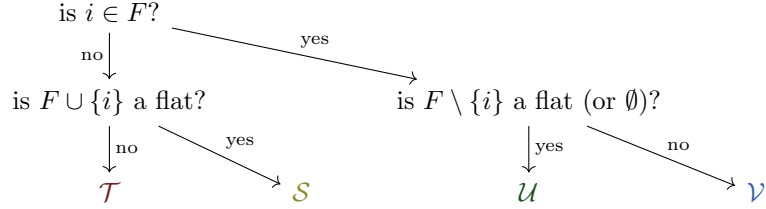
\[\begin{tikzcd}
	{\text{is $i \in F$?}} \\
	{\text{is $F\cup \{i\}$ a flat?}} && {\text{is $F\setminus\{i\}$ a flat (or $\emptyset$)?}} \\
	\textcolor{rgb,255:red,118;green,30;blue,30}{\mathcal{T}} & \textcolor{rgb,255:red,139;green,139;blue,35}{\mathcal{S}} & \textcolor{rgb,255:red,20;green,82;blue,20}{\mathcal{U}} & \textcolor{rgb,255:red,50;green,91;blue,179}{\mathcal{V}}
	\arrow["{\text{no}}"', from=1-1, to=2-1]
	\arrow["{\text{yes}}", from=1-1, to=2-3]
	\arrow["{\text{no}}", from=2-1, to=3-1]
	\arrow["{\text{yes}}", from=2-1, to=3-2]
	\arrow["{\text{yes}}", from=2-3, to=3-3]
	\arrow["{\text{no}}", from=2-3, to=3-4]
\end{tikzcd}\]

The sets $\mathcal{T}, \mathcal{S}, \mathcal{U}, \mathcal{V}$ are disjoint and partition $\mathcal{L}(M)$. 

Denote also $\mathcal{U}\setminus i := \{F\setminus \{i\} | F \in \mathcal{U} \}$ and 
$\mathcal{V}\setminus i := \{F\setminus \{i\} | F \in \mathcal{V}\}$.

Note the following reasons for this partition:
\begin{itemize}
    \item $\mathcal{U} \cup \mathcal{V}$ consists of all flats containing the deleted element $i$. Among these, the ones in $\mathcal{U}$ are the flats that decrease in rank when $i$ is deleted and the ones in $\mathcal{V}$ are the flats that remain the same rank when $i$ is deleted.

    \item $S \cup \mathcal{V}\setminus i = \mathcal{L}(M/i)$ are the flats of the contraction $M/i$.
    
    \item $\mathcal{T} \cup \mathcal{S} \cup \mathcal{V} \setminus i = \mathcal{L}(M \setminus i)$ are the flats of the deletion $M\setminus i$.
    
    \item $\mathcal{U}\setminus i = \mathcal{S} \cup \{\emptyset\}$.
\end{itemize}
This partition determines the behavior of the divisors under the map $\bar{\delta}^*$: 
\begin{itemize}
    \item If $F \in \mathcal{T}$, then $\bar{\delta}^*(y_F) = x_F$. 
    \item If $F \in \mathcal{S}$, then $\bar{\delta}^*(y_F) = x_F + x_{F \cup \{i\}}$.
    \item If $F \in \mathcal{V} \setminus i$, then $\bar{\delta}^*(y_{F}) = x_{F \cup i}$.
    \item The pullback of the special element $y_i = -\sum_{F \in \mathcal{V}\setminus i} y_F$ is
    $$\bar{\delta}^*(y_i) = \bar{\delta}^*\left(-\sum_{F \in \mathcal{V}\setminus i} y_F\right) = -\sum_{F \in \mathcal{V}} x_{F} = \sum_{F \in \mathcal{U}} x_{F}$$.
\end{itemize}
Under the partition, the resulting quotient lattice structure is
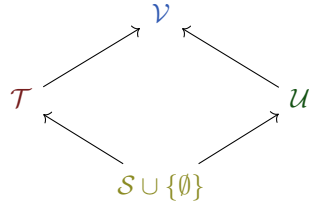
\[\begin{tikzcd}
	& \textcolor{rgb,255:red,50;green,91;blue,179}{\mathcal{V}} \\
	\textcolor{rgb,255:red,118;green,30;blue,30}{\mathcal{T}} && \textcolor{rgb,255:red,20;green,82;blue,20}{\mathcal{U}} \\
	& \textcolor{rgb,255:red,139;green,139;blue,35}{\mathcal{S} \cup \{\emptyset\}}
	\arrow[from=2-1, to=1-2]
	\arrow[from=2-3, to=1-2]
	\arrow[from=3-2, to=2-1]
	\arrow[from=3-2, to=2-3]
\end{tikzcd}\]

For ranks $r=1, \dots, d+1$, let 
\begin{align*}
    t_r := \sum_{\substack{F \in \mathcal{T} \\ \rk(F) = r}} x_F, \quad &v_r := \sum_{\substack{F \in \mathcal{V} \\ \rk(F) = r}} x_F,\\
    s_r := \sum_{\substack{F \in \mathcal{S} \\ \rk(F) = r}} x_F, \quad & u_r := \sum_{\substack{F \in \mathcal{U} \\ \rk(F) = r}} x_F.
\end{align*}
One of $\mathcal{T}, \mathcal{S}, \mathcal{U}$ and $\mathcal{V}$ may be empty in some rank $r$, in which case the sum is zero. For short hand, we also group indices with the same rank $(t + s)_r := t_r + s_r$, etc.
In particular, note that 
\begin{align*}
    \bar{\delta}^*\left( \sum_{\substack{F \in \mathcal{S} \\ \rk(F) = j}} y_F \right) & = s_j + u_{j+1} , &
    \bar{\delta}^*\left( \sum_{\substack{F \in \mathcal{T} \\ \rk(F) = j}} y_F \right) &= t_j \\
    \bar{\delta}^*\left( \sum_{\substack{F \in \mathcal{V} \setminus i \\ \rk(F) = j}} y_F \right) &= v_j
\end{align*}

Using this notation,
\begin{align*}
    \st(M) &= \prod_{r=1}^{d+1} \left(1 - \sum_{j=r}^{d+1}(s + t + u + v)_j \right)
\end{align*}
and
\begin{align*}
    \bar{\delta}^*(\st(M\setminus i)) & = \prod_{r=1}^{d+1} \left(1 - \bar{\delta}^* \left( \sum_{\substack{F \in \mathcal{L}(M \setminus i) \\ \rk(F) \geq r}} y_F \right) \right) \\
    & = \prod_{r=1}^{d+1} \left(1 - \bar{\delta}^* \left( \sum_{\substack{F \in \mathcal{S} \\ \rk(F) \geq r}} y_F + \sum_{\substack{F \in \mathcal{T} \\ \rk(F) \geq r}} y_F + \sum_{\substack{F \in \mathcal{V} \setminus i \\ \rk(F) \geq r}} y_F\right) \right)\\
    & = \prod_{r=1}^{d+1} \left(1 - \sum_{j \geq r} (s_j + u_{j+1}) + \sum_{j \geq r} t_j + \sum_{j \geq r} v_j \right)\\
    &= \prod_{r=1}^{d+1} \left(1 - (s + t + v)_r - \sum_{j=r+1}^{d+1}(s + t + u + v)_j\right)
\end{align*}

Note that $u_r$ is omitted from the $r$-th factor. We compute the difference $\st(M) - \bar{\delta}^*(\st(M\setminus i))$ and show it is equal to the additive inverse of the pullback to $A^*(M)$ of the staircase of $M / i$; concretely, denoting the rank function of $M$ by $\rk$, the rank function of $M / i$ by $\rk'$ and noting that for $F \in \mathcal{V} \setminus i$, $\rk(F) = \rk'(F)+1$,
\begin{align*}
    \bar{\delta}^* (y_i \cdot \iota(\st(M/i))) &= \bar{\delta}^* (y_i) \cdot \bar{\delta}^*\iota(\st(M/i)))\\
    &= \left(\sum_{r=1}^{d+1} u_i \right)\bar{\delta}^*\iota\left( \prod_{r=1}^{d}\left( 1 - \sum_{\substack{F \in \mathcal{L}(M / i) \\ \rk'(F) \geq r}} z_F \right) \right) \\
    &= \left(\sum_{r=1}^{d+1} u_i \right)\prod_{r=1}^{d}\left( 1 - \bar{\delta}^*\iota\left(\sum_{\substack{F \in \mathcal{S} \\ \rk(F) \geq r}} z_F\right) - \bar{\delta}^*\iota\left(\sum_{\substack{F \in \mathcal{V} \setminus i \\ \rk(F) \geq r + 1}} z_F \right)\right) \\
    &= \left(\sum_{r=1}^{d+1} u_i \right)\prod_{r=1}^{d}\left( 1 - \bar{\delta}^*\iota\left(\sum_{\substack{F \in \mathcal{S} \\ \rk(F) \geq r}} z_F\right) - \bar{\delta}^*\iota\left(\sum_{\substack{F \in \mathcal{V} \setminus i \\ \rk(F) \geq r + 1}} z_F \right)\right)\\
    &= \left(\sum_{r=1}^{d+1} u_i \right)\prod_{r=1}^{d}\left( 1 - \bar{\delta}^*\left(\sum_{\substack{F \in \mathcal{S} \\ \rk(F) \geq r}} y_F\right) - \bar{\delta}^*\left(\sum_{\substack{F \in \mathcal{V} \setminus i \\ \rk(F) \geq r + 1}} y_F \right)\right) \\
    &= \left(\sum_{r=1}^{d+1} u_i \right) \prod_{r=1}^{d}\left( 1 - \sum_{j \geq r} (s_j + u_{j+1}) +\sum_{j \geq r} v_j  \right) \\
    & =\left(\sum_{r=1}^{d+1} u_i \right) \prod_{r=1}^{d}\left( 1 - s_r - \sum_{j=r+1}^{d+1} (s + u + v)_{j} \right)
\end{align*}

The difference is:
\begin{align*}
    &\st(M) - \bar{\delta}^*(\st(M\setminus i)) =\\
    & - \sum_{k=1}^{d+1} \prod_{r = 1}^{k-1}\left(1 - (t+s+v)_r - \sum_{j=r+1}^{d+1} (t+s+u+v)_j \right) u_k \prod_{r=k+1}^{d+1}\left( 1 - \sum_{j=r}^{d+1}(t+s+u+v)_j\right)
\end{align*}

Note in the $k$-th term is like $\st(M)$ except the $k$-th factor is replaced by $-u_k$ and for each $r < k$, $u_{k}$ is omitted from the $r$-th factors. Since all pairs of flats in $\mathcal{T}$ and $\mathcal{U}$ are not comparable, $t_j u_l = 0$ for all $j$ and $l$. Since each term has some $u_k$ as a factor, all `$t$'s may be removed.

\begin{align*}
    &\st(M) - \bar{\delta}^*(\st(M\setminus i)) =\\
    & - \sum_{k=1}^{d+1} \prod_{r = 1}^{k-1}\left(1 - (s+v)_r - \sum_{j=r+1}^{d+1} (s+u+v)_j \right) u_k \prod_{r=k+1}^{d+1}\left( 1 - \sum_{j=r}^{d+1}(s+u+v)_j\right)
\end{align*}

Also $u_k s_l = 0$ if $k \leq l$. Thus:
\begin{align*}
    &\st(M) - \bar{\delta}^*(\st(M\setminus i)) =\\
    & - \sum_{k=1}^{d+1} \prod_{r = 1}^{k-1}\left(1 - (s+v)_r - \sum_{j=r+1}^{d+1} (s+u+v)_j \right) u_k \prod_{r=k+1}^{d+1}\left( 1 -s_{r-1} - \sum_{j=r}^{d+1}(s+u+v)_j\right)
\end{align*}

Also $u_k v_l = 0 $ if $l < k$. Then, for each term, index $k$, for all $l < k$ remove $v_{l}$ from the first $l$-th factors (before each $-u_k$).
\begin{align*}
    &\st(M) - \bar{\delta}^*(\st(M\setminus i)) \\
    =& - \sum_{k=1}^{d+1} \prod_{r = 1}^{k-1}\left(1 - s_r - \sum_{j=r+1}^{d+1} (s+u+v)_j \right) u_k \prod_{r=k+1}^{d+1}\left( 1 -s_{r-1} - \sum_{j=r}^{d+1}(s+u+v)_j\right)\\
    =& - \sum_{k=1}^{d+1} u_k \prod_{r = 1}^{k-1}\left(1 - s_r - \sum_{j=r+1}^{d+1} (s+u+v)_j \right)  \prod_{r=k}^{d}\left( 1 - s_r - \sum_{j=r+1}^{d+1}(s+u+v)_j\right) \\
    =& - \left( \sum_{k=1}^{d+1} u_k \right) \prod_{r = 1}^{d}\left(1 - s_r - \sum_{j=r+1}^{d+1} (s+u+v)_j \right)\\
    =& - \bar{\delta}^*(y_i \cdot \iota(\st(M / i))),
\end{align*}
which proves the claim.
\end{proof}

\bibliographystyle{alpha}
\bibliography{main.bib}

\end{document}